\documentclass[11pt]{amsart}

\usepackage[utf8]{inputenc}
\usepackage[T1]{fontenc}
\usepackage{amsmath,amssymb,amsthm,mathtools}
\usepackage{enumitem}
\usepackage{hyperref}
\usepackage{microtype}
\theoremstyle{plain}
\newtheorem{theorem}{Theorem}[section]
\newtheorem{corollary}[theorem]{Corollary}
\newtheorem{proposition}[theorem]{Proposition}

\newtheorem{definition}[theorem]{Definition}

\theoremstyle{remark}
\newtheorem{remark}[theorem]{Remark}

\newcommand{\Ric}{\operatorname{Ric}}
\newcommand{\tr}{\operatorname{tr}}

\newcommand{\rank}{\operatorname{rank}}
\newcommand{\Met}{\operatorname{Met}}
\newcommand{\M}{\mathcal M}

\newcommand{\mcT}{\mathcal T}
\newcommand{\R}{\mathbb R}

\title[Critical metrics for Dirac and Yamabe ratio of eigenvalues]{Critical metrics for the Dirac--Yamabe eigenvalue ratio and Einstein--Dirac equations }
\author{Mehdi Eddaoudi}
\address{Department of Mathematics and Data Science, Vrije Universiteit Brussel.
}
\email{mehdi.eddaoudi@vub.be}

\begin{document}

\begin{abstract}
Let $(M^n,g)$ be a closed spin manifold of dimension $n\geq 3$ with
positive Yamabe constant. We prove that critical metrics for the mixed eigenvalue ratio $\mathcal H(g)=\frac{\lambda_1(D_g)^2}{\lambda_1(L_g)},$
where $D_g$ is the Dirac operator and $L_g$ is the
conformal Laplacian are characterized by a harmonic section whose
energy-momentum tensor satisfies an Einstein--Dirac equation. We discuss consequences of this result for Hijazi's inequality, and we construct an example to prove that $\mathcal H$ is unbounded above within the conformal class of the round sphere.
\end{abstract}

\maketitle
\tableofcontents

\section{Introduction}
\subsection{The Dirac--Yamabe ratio of eigenvalues}

Let $(M^n,g)$ be a smooth closed connected spin Riemannian manifold of
dimension $n\geq 3$. We denote by $D_g$ the Dirac operator acting on the
complex spinor bundle $\Sigma_gM$, and by $L_g$ the conformal Laplacian, also known as the Yamabe operator, defined as
\[
    L_g=\Delta_g+c_nS_g,\qquad c_n=\frac{n-2}{4(n-1)},
\]
where $\Delta_g$ is the non-negative Laplacian and $S_g$ is the scalar
curvature.

These two operators admit a discrete sequence of eigenvalues, and we denote by
$\lambda_1(L_g)$ the first eigenvalue of $L_g$ and by $\lambda_1(D_g)$ the
first positive eigenvalue of $D_g$.

This paper is a follow-up to our previous result in \cite[Theorem~1.1]{eddaoudi2026harmoniceinsteinmapscriticalmetrics},
where we established a correspondence between critical metrics for the ratio $\lambda_2(L_g)/\lambda_1(L_g)$ and harmonic- Einstein maps into spheres.
The proof consists in computing the Euler--Lagrange equations via the Clarke analysis of Petrides and Tewodrose \cite{PetridesTewodrose2024}, and then splitting the resulting tensorial identity into its trace and trace-free parts. This suggests a more general principle for ratios of eigenvalues of conformally covariant operators, namely that their critical points may be characterized by a tensorial identity whose trace carries the conformal constraint, while its trace-free part encodes an explicit geometric interaction.

In this work we illustrate this idea by investigating critical points of the ratio of mixed eigenvalues of the Dirac operator and the conformal Laplacian
\[
    \mathcal{H}(g) := \frac{\lambda_1(D_g)^2}{\lambda_1(L_g)}.
\]
Since $D_g$ is a first-order operator and $L_g$ a second-order operator, the ratio
$\mathcal{H}$ is scale-invariant and well defined on the set of smooth metrics $\mathcal{M}(M)$ on $M$ having a positive Yamabe
constant,
\[
    \mathcal{M}_+ = \bigl\{g\in\mathcal{M}(M) : \lambda_1(L_g)>0\bigr\}.
\]

The choice of this ratio is also motivated by 
Hijazi's celebrated inequality \cite{Hijazi1986}
\begin{equation}\label{eq:hijazi}
    \frac{\lambda_1(D_g)^2}{\lambda_1(L_g)}
    \geq
    \frac{n}{n-2}=\frac{\lambda_1(D_{g_0})^2}{\lambda_1(L_{g_0})},
\end{equation}
where $(\mathbb{S}^n,g_0)$ is the standard sphere. Equality is
attained by metrics carrying a non-zero real Killing spinor, which is a
strictly stronger condition than being Einstein.

The relationship between the Dirac operator and the conformal
Laplacian is a central theme in spin geometry. For instance,
the Lichnerowicz formula
\[
    D_g^2=\nabla^*\nabla+\frac14S_g
\]
exhibits the square of the Dirac operator as a spinorial
counterpart of $L_g$. Moreover, at the variational level, minimizing
$\lambda_1(D_g)\operatorname{Vol}(M,g)^{1/n}$ over a conformal class
for a fixed spin structure is a spinorial analogue of the Yamabe
problem, which has motivated a vast literature on nonlinear
Dirac equations and sharp spinorial Sobolev inequalities
\cite{Lott1986,Baer1992,Ammann2003,AmmannHumbertMorel2006,
AmmannHumbert2006Tori,AmmannHumbert2006Overview,
AmmannGrosjeanHumbertMorel2008,Ammann2009,
Raulot2009,SireXu2021,IsobeXu2023,JulioBatalla2025,GrosjeanGinoux,Ginoux2009}.

By scale invariance, our problem  can be viewed as looking for critical metrics of $\lambda_1(D_g)^2$ under the normalization $\lambda_1(L_g)=1$. In our main Theorem \ref{thm:main} below, the trace part of the tensorial equation
corresponds to conformal variations and yields harmonic sections, while the
trace-free part couples the Einstein tensor to the Dirac stress--energy tensor. The trace part is reminiscent of the ideas of Karpukhin, M\'etras and Polterovich
\cite{KarpukhinMetrasPolterovich2024} on surfaces where they established that critical points in a conformal class of the area-normalized Dirac eigenvalue $\lambda_k(D_g)\operatorname{Area}(M,g)^{1/2}$ are given by harmonic maps into complex projective spaces.

Finally, we also mention the recent variational framework of Martynyuk \cite{Martynyuk2026}, who developed a criterion for the existence of minimizers of volume-normalized Dirac eigenvalues within a fixed conformal class, with an emphasis on the case of surfaces and with applications to the two-dimensional sphere.

\subsection{Main results and consequences}

Let $S(\Sigma_gM^{\oplus m})$ denote the unit sphere bundle of the direct sum
of $m$ copies of $\Sigma_gM$. For a harmonic section $\psi=(\psi_1,\ldots,\psi_m)$, we
denote by $Q_g$ the spinorial energy-momentum tensor defined in Section \ref{sec:harmonicsection}, and we write $Q_g(\psi) :=\sum_{\alpha=1}^mQ_g(\psi_\alpha)$.
We prove the following.

\begin{theorem}\label{thm:main}
A metric $g\in\M_+$ is Clarke critical for $\mathcal H$ if and only if there
exist a positive first eigenfunction $u>0$ of $L_g$ and a finite family of
pairwise $L^2(g)$-orthogonal spinors
$\varphi_1,\ldots,\varphi_m \in\ker(D_g-\lambda_1(D_g))$ such that
\[
\psi=\left(\frac{\varphi_1}{u},\ldots,\frac{\varphi_m}{u}\right)
:M\longrightarrow S(\Sigma_gM^{\oplus m})
\]
is a harmonic section with respect to $g_u=u^{4/(n-2)}g$ and satisfies
\begin{equation}\label{eq:einstein-dirac-section}
    Q_{g}(\psi)
    =
    c_n\frac{\lambda_1(D_g)}{\lambda_1(L_g)}\Ric_{g_u}.
\end{equation}
\end{theorem}

By taking the trace in the  tensorial identity \eqref{eq:einstein-dirac-section} we obtain
$\tr_{g_u}Q_g(\psi)=\lambda_1(D_g)u^{-4/(n-2)}$, therefore we can write 
\begin{equation}\label{eq:einstein-trace-reversed}
    \Ric_{g_u}-\frac12 S_{g_u}g_u
    =
    \frac{\lambda_1(L_g)}{c_n\lambda_1(D_g)}
    \left(
     Q_{g}(\psi)
    -
    \frac12
    \tr_{g_u}\bigl( Q_{g}(\psi)\bigr)g_u
    \right).
\end{equation}
Since the left-hand side is divergence-free, the energy-momentum part of $\psi$
satisfies the conservation law
\begin{equation}\label{eq:conservation}
    \operatorname{div}_{g_u}
    \left[
     Q_{g}(\psi)
    -
    \frac12
    \tr_{g_u}\bigl( Q_{g}(\psi)\bigr)g_u
    \right]
    =
    0.
\end{equation}

From the viewpoint of mathematical physics,
Equation~\eqref{eq:einstein-trace-reversed}
can be interpreted as a nonlinear Riemannian
Einstein--Dirac system, related to the framework
of Friedrich and Kim \cite{FriedrichKim2000}, 
Soler \cite{Soler1970} and Branding \cite{Branding2019}.

Theorem~\ref{thm:main} also yields several consequences for critical point in Hijazi's
inequality beyond the equality case. For instance, Corollary~\ref{cor:deffect}
expresses the defect $\lambda_1(D_g)^2-\frac{n}{n-2}\lambda_1(L_g)$ of critical
metrics in terms of the twistor energy of the conformally rescaled spinors.
When this defect vanishes, $u$ is forced to be constant and we recover real
Killing spinors. In Corollary~\ref{cor:anisotropic-hijazi} we refine this estimate through a weighted comparison of $|\Ric_{g_u}|_{g_u}^2$ and $S_{g_u}^2$. In
particular, if $\rank(\Ric_{g_u})\leq r$ for some $2\leq r\leq n$, we obtain in Corollary~\ref{cor:finiterank} the following rank estimate

\begin{equation}
    \frac{\lambda_1(D_g)^2}{\lambda_1(L_g)}
    \geq \frac{r(n-1)}{(r-1)(n-2)}.
\end{equation}
In particular when $r=n$, this recovers Hijazi's bound. We
characterize the equality case and we highlight examples of critical metrics attaining
equality.
\medskip

Our second result shows that $\mathcal H$ is unbounded above,
even within the round conformal class of the sphere $(\mathbb{S}^n,[g_0])$.

\begin{theorem}\label{thm:unbounded}
Let $(\mathbb S^n,g_0)$, $n\geq 3$, be the round sphere with its standard spin
structure. Then
\[
        \sup_{g\in[g_0]}
        \frac{\lambda_1(D_g)^2}{\lambda_1(L_g)}=+\infty.
\]
\end{theorem}

The construction of the example relies on a family of conformal metrics degenerating along a
totally geodesic subsphere of codimension two $\Sigma=\mathbb S^{n-2}\subset \mathbb S^n\subset\R^{n-1}\oplus\R^2$. While the first eigenvalue of
the conformal Laplacian tends to zero, the first Dirac eigenvalue retains a uniform gap
away from zero. The latter follows from the geometry of the limiting singular
metric, which is isometric to $\mathbb H^{n-1}\times \mathbb S^1$. More
precisely, let
$r(x,z)=|z|$, and define
\[
        f_\varepsilon=(r^2+\varepsilon^2)^{-1/2},
        \qquad
        g_\varepsilon=f_\varepsilon^2g_0.
\]
Then we prove that there is a constant $C>0$, independent of $\varepsilon$, such that
\[
        \lambda_1(L_{g_\varepsilon})\leq \frac{C}{\log(1/\varepsilon)},
        \qquad
        \lambda_1(D_{g_\varepsilon})^2\geq \frac14,
\]
and hence $\mathcal H(g_\varepsilon)\to +\infty$ as $\varepsilon\to 0$.

\subsection*{Organization of the paper}
Section~\ref{sec:harmonicsection} introduces the necessary material of spinorial energy-momentum
tensor and the weighted harmonic sections needed for our setting.
In Section~\ref{sec:variations} we compute the Euler--Lagrange equation
for $\mathcal H$, and then we prove Theorem~\ref{thm:main} in Section~\ref{sec:main-proof}.
Section~\ref{sec:consequences} establishes the consequences
for critical points in Hijazi's inequality beyond the equality case. 
Finally, Section~\ref{sec:unboundeness} proves that  in general the ratio $\mathcal{H}$ is unbounded in a given conformal class, which is
Theorem~\ref{thm:unbounded}.

\section{Harmonic sections of unit sphere bundles}\label{sec:harmonicsection}
\subsection{Energy-momentum Dirac tensor} We first recall some natural objects associated to the Dirac operator. 
 For a spinor
$\varphi\in\Gamma(\Sigma_gM)$, the spinorial energy-momentum tensor is the symmetric two-tensor defined by
\begin{equation}\label{eq: tensor Q}
    Q_g(\varphi)(X,Y)
    =
    \frac12\operatorname{Re}
    \big\langle
    X\cdot\nabla^g_Y\varphi
    +
    Y\cdot\nabla^g_X\varphi,
    \varphi
    \big\rangle_g .
\end{equation}

We denote by $S(\Sigma_gM^{\oplus m})$ the unit sphere bundle of
$\Sigma_gM^{\oplus m}$, and we set for a map of spinors
$\psi=(\psi_1,\ldots,\psi_m)\in\Gamma(\Sigma_gM^{\oplus m})$,
\[
    Q_g(\psi):=\sum_{\alpha=1}^m Q_g(\psi_\alpha).
\]

If $h=f^2g$ with $f>0$, we use the standard fiberwise
isometric identification of $\Sigma_gM$ and $\Sigma_hM$.
Under this identification, for any smooth real-valued
function $a$ and spinor $\phi$, we have
\[
Q_g(a\phi)=a^2Q_g(\phi),
\qquad
Q_h\bigl(f^{-\frac{n-1}{2}}\phi\bigr)
=f^{-(n-2)}Q_g(\phi).
\]
Both identities follow from the conformal transformation
of the spin connection and the skew-Hermitian character
of Clifford multiplication.

\subsection{Weighted harmonic sections} Now we introduce the right notion of harmonicity in unit sphere bundles that we use throughout this paper. In what follows, the spinor bundle $\Sigma_gM$, its Hermitian metric $\langle\cdot,\cdot\rangle_g$ and the spin connection $\nabla^g$ are those of the background metric $g$.
\begin{definition}\label{def: spinorial section}
Let \(h\) be a metric conformal to \(g\). We say that a section
\[
    \psi:M\longrightarrow S(\Sigma_gM^{\oplus m})
\]
is harmonic with respect to \(h\) if
\[
    \left((\nabla^{g})^*_{h}
    \nabla^{g}\psi\right)^\top=0,
\]
where \((\nabla^{g})^*_{h}\) denotes the formal adjoint of \(\nabla^g\)
computed with respect to the metric \(h\) on \(M\) and the Hermitian metric
induced by \(g\) on \(\Sigma_gM^{\oplus m}\). The superscript \(\top\)
denotes the orthogonal projection onto the tangent space of the unit sphere
bundle.
\end{definition}
This definition is a
weighted form of harmonic sections developed in \cite{Wood1997,BenyounesLoubeauWood2007}. More precisely, they are critical points of the weighted Dirichlet energy. We write the proof of this statement for the sake of completeness.
\begin{proposition}
    Let $h$ be a metric conformal to $g$. A section $\psi:M\longrightarrow S(\Sigma_gM^{\oplus m})$
    is harmonic with respect to $h$ if and only if it is a critical point of the weighted Dirichlet energy
    \[
       E_h(\psi)
      :=
      \frac12
      \int_M
       h^{ij}
       \left\langle
       \nabla^g_i\psi,\nabla^g_j\psi
       \right\rangle_g
    \,dv_h
      =
      \frac12\int_M |\nabla^g\psi|^2_{h,g}\,dv_h .
    \]
    In particular, for $h=g_u=u^{4/(n-2)}g$, we have
    \[
       E_{g_u}(\psi)=\frac12\int_Mu^2\,|\nabla^g\psi|_g^2\,dv_g .
    \]
\end{proposition}
Here the cotangent variables are contracted with \(h\), while the spinor
inner product is the one induced by \(g\).

\begin{proof}
Let $\psi_t$ be a variation through unit sections with $\dot\psi_0=\eta$, so
that $\operatorname{Re}\langle\eta,\psi\rangle_g=0$. Conversely, every such
$\eta$ arises this way, for instance with
$\psi_t=(\psi+t\eta)/|\psi+t\eta|_g$. An integration by parts gives
\[
    \left.\frac{d}{dt}\right|_{t=0}E_h(\psi_t)
    =\int_M\operatorname{Re}\big\langle
    (\nabla^g)^*_h\nabla^g\psi,\eta\big\rangle_g\,dv_h,
\]
which vanishes for all such $\eta$ if and only if
$\big((\nabla^g)^*_h\nabla^g\psi\big)^\top=0$. The expression of $E_{g_u}$
follows from $g_u^{-1}=u^{-4/(n-2)}g^{-1}$ and
$dv_{g_u}=u^{2n/(n-2)}dv_g$.
\end{proof}

\section{First variations and Clarke criticality}\label{sec:variations}

In this section we compute the first variations of the eigenvalues involved in
$\mathcal H$, and we identify the corresponding notion of criticality. Notice
that the first eigenvalue $\lambda_1(L_g)$ is simple, hence smooth with respect
to the metric, whereas the first positive Dirac eigenvalue $\lambda_1(D_g)$ may
have multiplicity. In that case, along a smooth deformation of the metric the
corresponding spectral cluster splits into several eigenvalue branches, and
$\lambda_1(D_g)$ is in general only locally Lipschitz, not differentiable. In
other terms, the ratio $\mathcal H$ does not admit a classical differential at
every metric, and the usual notion of critical point has to be replaced by a
weaker one. We therefore rely on the framework of Petrides and Tewodrose
\cite{PetridesTewodrose2024}, based on the Clarke generalized differential of a
locally Lipschitz functional, which we recall below.

\subsection{Clarke critical metrics}
Fix an integer $s\geq3$. Let $\Met^s(M)$ be the open subset
of positive definite tensors in the Banach space
$\mathcal C^s(S^2T^*M)$, where we identify
\[
T_g\Met^s(M)=\mathcal C^s(S^2T^*M).
\]
The functional $\mathcal H$ extends to the open set
\[
\mathcal M_+^s
=
\left\{
g\in\Met^s(M):\lambda_1(L_g)>0
\right\}.
\]
 For a locally Lipschitz functional $F$ on $\mathcal M_+^s$
and $h\in\mathcal C^s(S^2T^*M)$, the generalized
directional derivatives are defined as
\[
F^\circ_+(g;h)
=
\limsup_{g'\to g,\,t\downarrow0}
\frac{F(g'+th)-F(g')}{t},
\]
\[
F^\circ_-(g;h)
=
\liminf_{g'\to g,\,t\downarrow0}
\frac{F(g'+th)-F(g')}{t}.
\]
This allows to introduce the Clarke sub/superdifferential
\[
\partial_C^-F(g)
=
\left\{
\ell\in\bigl(\mathcal C^s(S^2T^*M)\bigr)^*:
\ell(h)\leq F^\circ_+(g;h)
\text{ for every }h
\right\},
\]
\[
\partial_C^+F(g)
=
\left\{
\ell\in\bigl(\mathcal C^s(S^2T^*M)\bigr)^*:
\ell(h)\geq F^\circ_-(g;h)
\text{ for every }h
\right\},
\]
where the star denotes the continuous dual.

\begin{definition}[Petrides and Tewodrose
\cite{PetridesTewodrose2024}]
A smooth metric $g\in\mathcal M_+$ is Clarke critical
for $F$ if
\[
0\in\partial_C^-F(g)\cup\partial_C^+F(g).
\]
\end{definition}

\subsection{Tensors associated with eigenvalue variations}

Let $u$ be a normalized positive first eigenfunction of $L_g$. Since $\lambda_1(L_g)$ is simple, along a metric variation $\dot g_0=h$, the first variation of the corresponding eigenvalue branch is represented by a symmetric two-tensor $\mcT^{L_g}_{\lambda_1(L_g)}(u)$:
\[
     \dot\lambda_1(L_g)[h]
     =
     \int_M\langle\mcT^{L_g}_{\lambda_1(L_g)}(u),h\rangle_g\,dv_g.
\]

\begin{proposition}\label{prop:firstvariationL_g}
The tensor $\mcT^{L_g}_{\lambda_1(L_g)}(u)$ is given by
\begin{align}\label{eq:TL}
\mcT^{L_g}_{\lambda_1(L_g)}(u)
={}&-du\otimes du+\frac12|du|_g^2g \nonumber \\
&+c_n\left(-u^2\Ric_g+\nabla_g^2(u^2)+\Delta_g(u^2)g +\frac12S_g u^2g\right)
-\frac{\lambda_1(L_g)}{2}u^2g .
\end{align}
\end{proposition}
The proof of Proposition \ref{prop:firstvariationL_g} is standard and can be found in \cite{PetridesTewodrose2024,eddaoudi2026harmoniceinsteinmapscriticalmetrics}.

Taking the trace and trace-free parts, a straightforward computation lead to
\begin{equation}\label{eq:TL-trace}
     \tr_g\mcT^{L_g}_{\lambda_1(L_g)}(u)=-\lambda_1(L_g) u^2,
\end{equation}
and
\begin{equation}\label{eq:TL-tf}
        \big(\mcT^{L_g}_{\lambda_1(L_g)}(u)\big)^0
        =
        -(du\otimes du)^0
        +c_n\left((\nabla_g^2(u^2))^0-u^2\Ric_g^0\right).
\end{equation}

In particular, if we set $g_u=u^{4/(n-2)}g$, then
\begin{equation}\label{eq:TL-Ric-full}
        -u^{-2}\mcT^{L_g}_{\lambda_1(L_g)}(u)=c_n\Ric_{g_u}, \qquad c_nS_{g_u}=\lambda_1(L_g)u^{-4/(n-2)}.
\end{equation}

For the Dirac operator, we use the formula of Bourguignon and Gauduchon
\cite{BourguignonGauduchon1992}, which computes the first variation of $\lambda_1(D_g)$
along each analytic eigenvalue branch.

\begin{proposition}[Bourguignon and Gauduchon's formula \cite{BourguignonGauduchon1992}]
Along the eigenvalue branch corresponding to the spinor $\psi$, the first variation is given by:
\[
\dot\lambda_1(D_g)[h]
=
-\frac12
\int_M
\langle Q_g(\psi),h\rangle_g\,dv_g,
\]
where the tensor $Q_g(\psi)$ is defined in  \eqref{eq: tensor Q}.
\end{proposition}

Therefore, the first variation of \(\lambda_1(D_g)^2\) along this branch is represented by $-\lambda_1(D_g)Q_g(\psi).$
According to the Clarke description for multiple eigenvalues, an element of the relevant Clarke superdifferential of \(\lambda_1(D_g)^2\) is obtained by taking the convex hull of these variations:
\[
-\lambda_1(D_g)\sum_{\alpha=1}^m c_\alpha Q_g(\psi_\alpha) \in \partial_C^+\big(\lambda_1(D_g)^2\big)(g),
\]
where \(c_\alpha\geq 0\), \(\sum_{\alpha=1}^m c_\alpha=1\), and the eigenspinors \(\psi_\alpha\) are normalized  in \(\ker(D_g-\lambda_1(D_g))\).

\subsection{Euler--Lagrange equations}
We now make use of the machinery of Petrides and Tewodrose
\cite[Theorems~3.13 and~3.14]{PetridesTewodrose2024} to compute Euler Lagrange equations for the ratio functional $\mathcal H$.

\begin{proposition}\label{prop:EL}
A metric $g\in\mathcal{M}_+$ is Clarke critical for $\mathcal H$ if and only if there exist a positive first $L^2$ normalized eigenfunction $u>0$ of $L_g$,
and finitely many pairwise $L^2(g)$-orthogonal spinors $ \varphi_1,\ldots,\varphi_m\in\ker(D_g-\lambda_1(D_g))$ such that
\begin{equation}\label{eq:EL}
        \sum_{\alpha=1}^m Q_g(\varphi_\alpha)
        =
        -\frac{\lambda_1(D_g)}{\lambda_1(L_g)}
        \mcT^{L_g}_{\lambda_1(L_g)}(u).
\end{equation}
\end{proposition}
\begin{proof}
Assume first that $g$ is Clarke critical for $\mathcal H$.
Choose the positive first eigenfunction $u$ of $L_g$ with
$\int_Mu^2\,dv_g=1$. Since $\lambda_1(L_g)$ is simple,
it depends smoothly on the metric.

For $F(x,y)=x/y$, the partial derivatives at
$(\lambda_1(D_g)^2,\lambda_1(L_g))$ are
\[
\partial_1F=\frac1{\lambda_1(L_g)},
\qquad
\partial_2F=-\frac{\lambda_1(D_g)^2}{\lambda_1(L_g)^2}.
\]
By Theorem~3.13(A) of \cite{PetridesTewodrose2024}
and the first-variation formulae, there exist
$L^2(g)$-orthonormal spinors
$\psi_\alpha\in\ker(D_g-\lambda_1(D_g))$
and coefficients $c_\alpha\geq0$, $\sum_\alpha c_\alpha=1$,
such that
\[
0=
-\frac{\lambda_1(D_g)}{\lambda_1(L_g)}
\sum_\alpha c_\alpha Q_g(\psi_\alpha)
-\frac{\lambda_1(D_g)^2}{\lambda_1(L_g)^2}
T^{L_g}_{\lambda_1(L_g)}(u).
\]
Setting $\varphi_\alpha=\sqrt{c_\alpha}\psi_\alpha$
and discarding zero terms gives \eqref{eq:EL}, with
pairwise $L^2(g)$-orthogonal spinors.

Conversely, assume \eqref{eq:EL}.
Taking its trace gives
$\sum_\alpha|\varphi_\alpha|_g^2=u^2$.
Hence, after discarding zero spinors, the quantities
\[
c_\alpha=\|\varphi_\alpha\|_{L^2(g)}^2,
\qquad
\psi_\alpha=\frac{\varphi_\alpha}{\sqrt{c_\alpha}}
\]
satisfy $\sum_\alpha c_\alpha=1$ and
$\|\psi_\alpha\|_{L^2(g)}=1$.
Equation~\eqref{eq:EL} therefore yields the same weighted
first-variation identity as above.

The denominator eigenvalue is simple, while the numerator
is at the beginning of its multiplicity block and
$\partial_1F>0$. Thus the sufficient condition of
\cite[Theorem~3.14]{PetridesTewodrose2024}
applies to the same ratio, and $g$ is Clarke critical.
\end{proof}

\section{Proof of the main Theorem \ref{thm:main}}\label{sec:main-proof}

\begin{proof}[Proof of Theorem~\ref{thm:main}]
Assume first that $g$ is Clarke critical. By Proposition~\ref{prop:EL},
there exist $u>0$ and eigenspinors $\varphi_\alpha$ such that
\begin{equation}\label{eq:tensor}
      \sum_{\alpha=1}^m Q_g(\varphi_\alpha)
      =
      -\frac{\lambda_1(D_g)}{\lambda_1(L_g)}
      \mcT^L_{\lambda_1(L_g)}(u).
\end{equation}
Taking the trace we obtain
\[
        \sum_{\alpha=1}^m|\varphi_\alpha|_g^2=u^2.
\]
Therefore $\psi=\left(\frac{\varphi_1}{u},\ldots,
\frac{\varphi_m}{u}\right)$ is a section of $S(\Sigma_gM^{\oplus m})$.

We first prove the harmonicity with respect to $g_u=u^{4/(n-2)}g$ of $\psi$. We start by applying Lichnerowicz formula
\[
      (\nabla^g)^*\nabla^g(u\psi_\alpha)
      +\frac14S_g u\psi_\alpha
       =
       \lambda_1(D_g)^2u\psi_\alpha,
\]
in which we expand the first term,
\[
      (\nabla^g)^*\nabla^g(u\psi_\alpha)
      =
       u(\nabla^g)^*\nabla^g\psi_\alpha
       +(\Delta_g u)\psi_\alpha
      -2\nabla^g_{\nabla^g u}\psi_\alpha.
\]
Hence
\[
            (\nabla^g)^*\nabla^g\psi_\alpha
            -2\nabla^g_{\nabla^g\log u}\psi_\alpha
                =
                \left(
              \lambda_1(D_g)^2-\frac14S_g-\frac{\Delta_gu}{u}
                \right)\psi_\alpha.
\]
Since the right-hand side is normal to the unit sphere bundle, and the formal adjoint with respect to $g_u$ satisfies
\[
     (\nabla^g)^*_{g_u}\nabla^g\psi
     =
       u^{-4/(n-2)}
      \left(
      (\nabla^g)^*\nabla^g\psi
      -2\nabla^g_{\nabla^g\log u}\psi
       \right),
\]
we deduce $\left((\nabla^g)^*_{g_u}\nabla^g\psi\right)^\top=0$. This proves the harmonicity of $\psi$.

We next consider the tensor equation. For every spinor $\varphi$,
\[
      Q_g\left(\frac{\varphi}{u}\right)=u^{-2}Q_g(\varphi).
\]
Indeed, the derivative of $u^{-1}$ produces terms containing
$\operatorname{Re}\langle X\cdot\varphi,\varphi\rangle_g$, which vanish.
Therefore, by \eqref{eq:tensor},
\[
      Q_g(\psi)
      =
      u^{-2}\sum_{\alpha=1}^m Q_g(\varphi_\alpha)
      =
      -\frac{\lambda_1(D_g)}{\lambda_1(L_g)}
     u^{-2}\mcT^L_{\lambda_1(L_g)}(u)
      =
      c_n\frac{\lambda_1(D_g)}{\lambda_1(L_g)}\Ric_{g_u}.
\]

Conversely, assume that such $u$ and $\varphi_\alpha$ are given by Theorem $\ref{thm:main}$. By
rescaling $u$ and the $\varphi_\alpha$ by the same positive constant, we may
assume $\int_Mu^2dv_g=1$. Since $\psi$ is spherical,
\[
     \sum_{\alpha=1}^m|\varphi_\alpha|_g^2=u^2.
\]
Using again $Q_g(\varphi_\alpha/u)=u^{-2}Q_g(\varphi_\alpha)$ and the tensor
equation, we compute
\[
      \sum_{\alpha=1}^mQ_g(\varphi_\alpha)
      =
     u^2Q_g(\psi)
     =
     c_n\frac{\lambda_1(D_g)}{\lambda_1(L_g)}u^2\Ric_{g_u}
     =
     -\frac{\lambda_1(D_g)}{\lambda_1(L_g)}
     \mcT^L_{\lambda_1(L_g)}(u),
\]
Proposition~\ref{prop:EL}
then implies that $g$ is Clarke critical.
\end{proof}

\section{Consequences for Hijazi's inequality}\label{sec:consequences}

 \subsection{The twistor operator} The harmonic spinorial section in Theorem \ref{thm:main} also provides a geometric interpretation of
the defect in Hijazi's inequality through the twistor operator \(P^{g}\)  defined as
\[
P^{g}_X\varphi
:=
\nabla^{g}_X\varphi
+
\frac1n X\cdot D_{g}\varphi.
\]
We also write $(P^g\varphi)(X)=P_X^g\varphi$, so that
$|P^g\varphi|_g^2=\sum_i|P_{e_i}^g\varphi|_g^2$
in a local orthonormal frame.

\begin{corollary}\label{cor:deffect}
Let \(g\in\M_+\) be Clarke critical for \(\mathcal H\), and let
\(u>0\), \(g_u=u^{4/(n-2)}g\), and $\psi=
\left(
\frac{\varphi_1}{u},\ldots,\frac{\varphi_m}{u}
\right)$ be given by Theorem~\ref{thm:main}. Then
\[
\begin{aligned}
&\sum_{\alpha=1}^m
\int_M
\left|
P^{g_u}
\left(
u^{-\frac1{n-2}}\psi_\alpha
\right)
\right|^2_{g_u}
\,dv_{g_u}
\\
&\qquad =
\frac{n-1}{n}\lambda_1(L_g)
\left(
\frac{\lambda_1(D_g)^2}{\lambda_1(L_g)}
-
\frac{n}{n-2}
\right)
\int_M u^{-\frac6{n-2}}\,dv_{g_u},
\end{aligned}
\]
\end{corollary}

In other terms, the
defect $\frac{\lambda_1(D_g)^2}{\lambda_1(L_g)}
-
\frac{n}{n-2}$
is exactly measured by the twistor energy of the associated spherical section of the critical metric $g$. Therefore the rigidity in Hijazi's inequality forces
$P^{g_u}
\left(
u^{-\frac1{n-2}}\psi_\alpha
\right)
=0$
for every $\alpha$, meaning that all the rescaled spinors $\chi_\alpha = u^{-1/(n-2)}\psi_\alpha$ are twistor spinors. Since they also satisfy
$D_{g_u}\chi_\alpha
=\lambda_1(D_g)u^{-2/(n-2)}\chi_\alpha$,
the twistor equation reduces to
\[
      \nabla^{g_u}_X \chi_\alpha
      =
      -\frac{\lambda_1(D_g)}{n}u^{-2/(n-2)} X\cdot \chi_\alpha.
\]
Taking the derivative of the squared norm $\sum_\alpha|\chi_\alpha|_{g_u}^2=u^{-2/(n-2)}$ and using the skew-Hermitian character of the Clifford multiplication immediately gives $d(u^{-2/(n-2)})=0$. Hence $u$ is constant, and the non-zero components $\chi_\alpha$ are real Killing spinors. This recovers the equality case of Hijazi’s inequality.

\begin{proof}[Proof of Corollary \ref{cor:deffect}]
Let us denote by $\chi_\alpha = u^{-1/(n-2)}\psi_\alpha$ the conformally rescaled spinors.
Under the conformal identification of spinor bundles, the Dirac operator transforms as
\[
     D_{g_u}(\chi_\alpha)
     =
     \lambda_1(D_g)u^{-2/(n-2)} \chi_\alpha,
\]
and the pointwise squared norm satisfies
\[
     \sum_{\alpha=1}^m |\chi_\alpha|_{g_u}^2
     =
     u^{-2/(n-2)} \sum_{\alpha=1}^m |\psi_\alpha|_g^2
      =
        u^{-2/(n-2)}.
\]
For any spinor $\chi$, the twistor decomposition implies
\[
     |\nabla^{g_u}\chi|_{g_u}^2
     = |P^{g_u}\chi|_{g_u}^2 + \frac1n|D_{g_u}\chi|_{g_u}^2.
\]
Integrating this identity over $M$ and using integration by parts in Lichnerowicz formula, we obtain
\[
     \int_M|P^{g_u}\chi|_{g_u}^2\,dv_{g_u}
      =
      \frac{n-1}{n}\int_M|D_{g_u}\chi|_{g_u}^2\,dv_{g_u}
     -
     \frac14\int_M S_{g_u}|\chi|_{g_u}^2\,dv_{g_u}.
\]
Applying this to each $\chi_\alpha$ and summing over $\alpha$, we deduce
\begin{align*}
\sum_{\alpha=1}^m
\int_M
|P^{g_u}(\chi_\alpha)|_{g_u}^2\,dv_{g_u}
&=
\frac{n-1}{n} \lambda_1(D_g)^2 \int_M u^{-4/(n-2)} \sum_{\alpha=1}^m |\chi_\alpha|_{g_u}^2 \,dv_{g_u} \\
&\quad - \frac14 \int_M S_{g_u} \sum_{\alpha=1}^m |\chi_\alpha|_{g_u}^2 \,dv_{g_u} \\
&=
\left(
\frac{n-1}{n}\lambda_1(D_g)^2
-
\frac{\lambda_1(L_g)}{4c_n}
\right)
\int_M u^{-6/(n-2)}\,dv_{g_u},
\end{align*}
where in the last step we used the conformal constraint $c_nS_{g_u}=\lambda_1(L_g)u^{-4/(n-2)}$. Hence the asserted identity is deduced after rearranging the terms. 
\end{proof}

\subsection{A refined rank Hijazi inequality}
The next two corollaries \ref{cor:anisotropic-hijazi} and \ref{cor:finiterank} are a rank-refined Hijazi inequality for critical metrics, in the spirit of \cite[Corollary 3.2]{eddaoudi2026harmoniceinsteinmapscriticalmetrics}.

\begin{corollary}
\label{cor:anisotropic-hijazi}
If \(g\in\M_+\) is Clarke critical for \(\mathcal H\), then
\[
      \frac1n
      \leq
      \frac{
      \displaystyle
      \int_M |\Ric_{g_u}|_{g_u}^2
      u^{\frac2{n-2}}\,dv_{g_u}
      }{
      \displaystyle
      \int_M S_{g_u}^2
      u^{\frac2{n-2}}\,dv_{g_u}
      }
      <1
\]
and
\[
      \frac{\lambda_1(D_g)^2}{\lambda_1(L_g)}
       \geq
       \frac{
      1
      }{
       4c_n
       \left(
      1-
      \displaystyle
      \frac{
      \int_M |\Ric_{g_u}|_{g_u}^2
       u^{\frac2{n-2}}\,dv_{g_u}
       }{
       \int_M S_{g_u}^2
      u^{\frac2{n-2}}\,dv_{g_u}
      }
      \right)
       }.
\]
\end{corollary}

\begin{proof}[Proof of Corollary~\ref{cor:anisotropic-hijazi}]
Let us denote again $\chi_\alpha = u^{-1/(n-2)}\psi_\alpha$. These conformally rescaled spinors satisfy
\[
      D_{g_u}(\chi_\alpha)
       =
       \lambda_1(D_g)u^{-2/(n-2)} \chi_\alpha,
\]
\[
       \sum_{\alpha=1}^m |\chi_\alpha|_{g_u}^2
       = u^{-2/(n-2)},
\]
and their spinorial energy-momentum tensor transforms as
\[
       \sum_{\alpha=1}^m Q_{g_u}(\chi_\alpha)
       = Q_g(\psi)
      = c_n\frac{\lambda_1(D_g)}{\lambda_1(L_g)}\Ric_{g_u}.
\]
We evaluate the pointwise squared norm of the energy-momentum tensor. By the Cauchy-Schwarz inequality, the norm of $Q_{g_u}$ acting on a family of spinors satisfies $|\sum_\alpha Q_{g_u}(\chi_\alpha)|_{g_u}^2 \le \big(\sum_\alpha |\nabla^{g_u} \chi_\alpha|_{g_u}^2\big) \big(\sum_\alpha |\chi_\alpha|_{g_u}^2\big)$. Substituting the expressions above yields
\[
       c_n^2\frac{\lambda_1(D_g)^2}{\lambda_1(L_g)^2}
       |\Ric_{g_u}|_{g_u}^2
       \leq
       \left( \sum_{\alpha=1}^m \left|\nabla^{g_u} \chi_\alpha \right|_{g_u}^2 \right) u^{-2/(n-2)}.
\]
Multiplying both sides by $u^{2/(n-2)}$ and integrating over $M$, we obtain
\[
    c_n^2\frac{\lambda_1(D_g)^2}{\lambda_1(L_g)^2}
    \int_M |\Ric_{g_u}|_{g_u}^2u^{2/(n-2)}\,dv_{g_u}
      \leq
       \int_M \sum_{\alpha=1}^m \left|\nabla^{g_u} \chi_\alpha \right|_{g_u}^2 \,dv_{g_u}.
\]
Using the Lichnerowicz formula and an integration by parts as in the proof of
the previous corollary, we obtain
\[
\begin{aligned}
    \int_M\sum_{\alpha=1}^m
    |\nabla^{g_u}\chi_\alpha|_{g_u}^2\,dv_{g_u}
    &=
    \int_M\sum_{\alpha=1}^m
    |D_{g_u}\chi_\alpha|_{g_u}^2\,dv_{g_u}
    -\frac14\int_M S_{g_u}
    \sum_{\alpha=1}^m|\chi_\alpha|_{g_u}^2\,dv_{g_u}\\
    &=
    \left(\lambda_1(D_g)^2-\frac{\lambda_1(L_g)}{4c_n}\right)
    \int_Mu^{-6/(n-2)}\,dv_{g_u},
\end{aligned}
\]
where we used $D_{g_u}\chi_\alpha
=\lambda_1(D_g)u^{-2/(n-2)}\chi_\alpha$,
$\sum_\alpha|\chi_\alpha|_{g_u}^2=u^{-2/(n-2)}$,
and $c_nS_{g_u}=\lambda_1(L_g)u^{-4/(n-2)}$.
On the other hand, we also have
\[
       \int_MS_{g_u}^2 u^{2/(n-2)}\,dv_{g_u}
    =
       \frac{\lambda_1(L_g)^2}{c_n^2}
       \int_Mu^{-6/(n-2)}\,dv_{g_u}.
\]
Substituting back into the last inequality, we obtain
\[
       \lambda_1(D_g)^2
       \int_M |\Ric_{g_u}|_{g_u}^2u^{2/(n-2)}\,dv_{g_u}
       \leq
       \left(\lambda_1(D_g)^2-\frac{\lambda_1(L_g)}{4c_n}\right)
       \int_MS_{g_u}^2u^{2/(n-2)}\,dv_{g_u},
\]
which is the stated lower bound. The strict inequality bounding the curvature ratio by $1$ comes from the fact that $\lambda_1(L_g) > 0$. The lower bound $1/n$ simply follows from the universal pointwise algebraic inequality $|\Ric_{g_u}|_{g_u}^2 \geq S_{g_u}^2/n$.

\end{proof}

\begin{corollary}\label{cor:finiterank}
Let $g\in\mathcal M_+$ be Clarke critical for $\mathcal H$, with
$u$, $g_u$ and $\varphi_\alpha$ as in Theorem~\ref{thm:main}.
If $\operatorname{rank}\Ric_{g_u}\leq r$ for some $2\leq r\leq n$, then
\begin{equation}\label{eq: rankrefined}
    \frac{\lambda_1(D_g)^2}{\lambda_1(L_g)}
    \geq \frac{r(n-1)}{(n-2)(r-1)}.
\end{equation}
Moreover, if equality holds, then $u$ is constant and the spinors
$\varphi_\alpha$ are Killing in the Ricci-positive directions
and parallel in the kernel of $\Ric_{g_u}$.
\end{corollary}
\begin{proof}
Since $\operatorname{rank}\Ric_{g_u}\leq r$, by the Cauchy–Schwarz inequality we have
\[
    |\Ric_{g_u}|_{g_u}^2\geq\frac1r S_{g_u}^2.
\]
Combining this with Corollary~\ref{cor:anisotropic-hijazi} yields 
$$ \frac{\lambda_1(D_g)^2}{\lambda_1(L_g)}
    \geq \frac{r(n-1)}{(n-2)(r-1)}.$$
Assume now that equality holds in \eqref{eq: rankrefined}, and let
$\chi_\alpha=u^{-1/(n-2)}\psi_\alpha$ be the rescaled spinors as in the preceding proof.
Then equality holds in both inequalities used above, and we obtain pointwise
\[
    |\Ric_{g_u}|_{g_u}^2=\frac1r S_{g_u}^2,
    \qquad
    \sum_{\alpha=1}^m|\nabla^{g_u}\chi_\alpha|_{g_u}^2
    =c_n^2\frac{\lambda_1(D_g)^2}{\lambda_1(L_g)^2}
      |\Ric_{g_u}|_{g_u}^2u^{2/(n-2)}.
\]
Since $S_{g_u}>0$, equality in the rank estimate implies that
$\Ric_{g_u}$ has rank exactly $r$ everywhere and that all its
nonzero eigenvalues equal $S_{g_u}/r$. Thus
\[
    \Ric_{g_u}^{\sharp}=\frac{S_{g_u}}r\Pi,
\]
where $\Pi$ is the $g_u$-orthogonal projection onto the image of the Ricci endomorphism $\Ric_{g_u}^{\sharp}$.

Next we define the symmetric endomorphism
\[
    A=c_n\frac{\lambda_1(D_g)}{\lambda_1(L_g)}
      u^{2/(n-2)}\Ric_{g_u}^{\sharp}.
\]
Using the identities for $\sum_\alpha|\chi_\alpha|_{g_u}^2$
and $\sum_\alpha Q_{g_u}(\chi_\alpha)$ established above,
by a direct expansion in a local $g_u$-orthonormal frame $(e_i)$ we have
\[
\begin{aligned}
    &\sum_{\alpha=1}^m\sum_{i=1}^n
    \left|\nabla_{e_i}^{g_u}\chi_\alpha
          +A(e_i)\cdot_{g_u}\chi_\alpha\right|_{g_u}^2\\
    &\qquad=
    \sum_{\alpha=1}^m|\nabla^{g_u}\chi_\alpha|_{g_u}^2
    -c_n^2\frac{\lambda_1(D_g)^2}{\lambda_1(L_g)^2}
      |\Ric_{g_u}|_{g_u}^2u^{2/(n-2)}
    =0.
\end{aligned}
\]
Therefore,
\[
    \nabla_X^{g_u}\chi_\alpha=-A(X)\cdot_{g_u}\chi_\alpha.
\]
By the skew-Hermitian character of Clifford multiplication,
\[
    X\left(\sum_{\alpha=1}^m|\chi_\alpha|_{g_u}^2\right)
    =-2\sum_{\alpha=1}^m
      \operatorname{Re}
      \langle A(X)\cdot_{g_u}\chi_\alpha,\chi_\alpha\rangle_{g_u}
    =0.
\]
Since $\sum_\alpha|\chi_\alpha|_{g_u}^2=u^{-2/(n-2)}$,
the function $u$ is constant.

Finally, using again the identity $c_nS_{g_u}=\lambda_1(L_g)u^{-4/(n-2)}$, we obtain
\[
    A=\frac{\lambda_1(D_g)}r u^{-2/(n-2)}\Pi.
\]
Under the conformal identification, the spin connections of $g$
and $g_u$ agree because $u$ is constant, while Clifford
multiplication satisfies $X\cdot_{g_u}=u^{2/(n-2)}X\cdot_g$.
Recalling that $\chi_\alpha=u^{-(n-1)/(n-2)}\varphi_\alpha$,
we therefore conclude that
\[
    \nabla_X^g\varphi_\alpha
    =-\frac{\lambda_1(D_g)}r\Pi(X)\cdot_g\varphi_\alpha.
\]
Thus the spinors are Killing in the Ricci-positive directions
and parallel in the kernel of $\Ric_{g_u}$.
\end{proof}

\begin{remark}
The constant in the rank-refined estimate
in Equation \eqref{eq: rankrefined} is sharp. Indeed, let
\(2\leq r\leq n\), let \((N^r,g_N)\) be a closed spin
manifold carrying a non-zero real Killing spinor with
Killing constant \(-\sqrt K/2\), and let
\((F^{n-r},g_F)\) be a closed Ricci-flat spin manifold
carrying a non-zero parallel spinor. When \(r\) is odd
and \(n-r>0\) is even, assume in addition that \(F\)
carries parallel spinors of both chiralities. The factor
\(F\) is omitted when \(r=n\).

On the product
\[
       M=N^r\times F^{n-r},
       \qquad
       g=g_N+g_F,
\]
the product construction then provides a unit spinor
\(\psi\) satisfying
\[
       \nabla^g_X\psi
       =
       -\frac{\sqrt K}{2}\,\Pi(X)\cdot\psi,
\]
where \(\Pi\) is the orthogonal projection onto \(TN\).
In the parity case specified above, the parallel spinor
on \(F\) is chosen with the chirality giving this sign.
Therefore,
\[
      D_g\psi=\frac{r\sqrt K}{2}\psi,
      \qquad
      Q_g(\psi)=\frac{\sqrt K}{2}\,g_N.
\]
The Killing spinor equation on \(N\) leads to
\[
       \Ric_g=(r-1)K\,g_N,
       \qquad
       S_g=r(r-1)K.
\]
Thus the Ricci tensor has rank \(r\), and
     $\lambda_1(L_g)=c_n r(r-1)K.$
Moreover, by Friedrich's inequality \cite{Friedrich1980} on \(N\) we obtain
\(D_N^2\geq r^2K/4\). Therefore the product formula for the square of the Dirac operator yields
\[
      D_g^2\geq\frac{r^2K}{4}.
\]
Since \(\psi\) realizes the positive eigenvalue
\(r\sqrt K/2\), we obtain
\[
      \lambda_1(D_g)=\frac{r\sqrt K}{2}.
\]

To verify that this metric is critical for the ratio $\mathcal{H}$, let
\(u=\operatorname{Vol}(M,g)^{-1/2}\) be the normalized
constant first eigenfunction of \(L_g\), and set
\(\varphi=u\psi\). Since
\[
      \mcT^{L_g}_{\lambda_1(L_g)}(u)
       =-c_nu^2\Ric_g,
\]
we have
\begin{align*}
      Q_g(\varphi)
       &=
       u^2\frac{\sqrt K}{2}\,g_N\\
    &=
       c_n\frac{\lambda_1(D_g)}{\lambda_1(L_g)}
       u^2\Ric_g\\
     &=
      -\frac{\lambda_1(D_g)}{\lambda_1(L_g)}
    \mcT^{L_g}_{\lambda_1(L_g)}(u).
\end{align*}
Proposition~\ref{prop:EL} shows that \(g\) is Clarke
critical for \(\mathcal H\). Therefore
\[
      \frac{\lambda_1(D_g)^2}{\lambda_1(L_g)}
       =
       \frac{r(n-1)}{(n-2)(r-1)}.
\]
Thus these product metrics attain equality in the
rank-refined bound. In particular, taking \(N\) to be
the round sphere of sectional curvature \(K\) and
\(F\) to be a flat torus with the trivial spin structure
provides such examples for every \(2\leq r\leq n\).
For related Dirac eigenvalue estimates on locally reducible
manifolds, see \cite{Alexandrov2007}.
\end{remark}

\section{Unboundedness of the ratio $\mathcal{H}$ in the round conformal class}\label{sec:unboundeness}

In this section we prove that the ratio  $\mathcal{H}= \frac{\lambda_1(D_g)^2}{\lambda_1(L_g)}$ cannot be bounded above in the round conformal class of the sphere. To do so, we introduce a specific family of metrics that pinches the geometry around an equatorial subsphere, and we study the asymptotic behavior of the conformal Laplacian and the Dirac operator separately. This construction allows us to construct a straightforward easy example.

\smallskip
 \subsection{Pinching a sphere along a codimension-two subsphere}
Let $S^n \subset \mathbb{R}^{n-1} \oplus \mathbb{R}^2$ and let $\Sigma = S^{n-2}$ be the totally geodesic equatorial subsphere defined by $z=0$, and let $r(x,z)=|z|=\sin d_{g_0}((x,z),\Sigma)$. For any $\epsilon > 0$, we define the conformal factor $f_\epsilon = (r^2 + \epsilon^2)^{-1/2}$ and the corresponding metric $g_\epsilon = f_\epsilon^2 g_0$.

\begin{proposition} \label{prop:laplacian_collapse}
There exists a constant $C > 0$, independent of $\epsilon$, such that the first eigenvalue of the conformal Laplacian satisfies
$$ \lambda_1(L_{g_\epsilon}) \le \frac{C}{\log(1/\epsilon)}. $$
\end{proposition}

\begin{proof}
Let $g_\varepsilon=u_\varepsilon^{4/(n-2)}g_0$,
where $u_\varepsilon=f_\varepsilon^{(n-2)/2}$.
By the conformal covariance of the conformal
Laplacian, taking $\varphi=u_\varepsilon^{-1}$ gives
\[
\int_{S^n}\varphi L_{g_\varepsilon}\varphi\,
dv_{g_\varepsilon}
=
\int_{S^n}L_{g_0}1\,dv_{g_0}
=
\frac{n(n-2)}{4}
\operatorname{Vol}(S^n,g_0).
\]
Moreover,
\[
\int_{S^n}\varphi^2\,dv_{g_\varepsilon}
=
\int_{S^n}f_\varepsilon^2\,dv_{g_0}
=
\int_{S^n}\frac{dv_{g_0}}{r^2+\varepsilon^2}.
\]
In Fermi coordinates around
$\Sigma=S^{n-2}$, the round volume element is
\[
dv_{g_0}
=
\cos^{n-2}(\rho)\sin(\rho)\,
d\rho\,d\theta\,dv_{S^{n-2}},
\]
where $r=\sin(\rho)$. Therefore, for any
fixed sufficiently small $\delta>0$,
there exists $c>0$, independent of
$\varepsilon$, such that
\[
\int_{S^n}\frac{dv_{g_0}}{r^2+\varepsilon^2}
\geq
c\int_0^\delta
\frac{r}{r^2+\varepsilon^2}\,dr
\geq c'\log(1/\varepsilon)
\]
for all sufficiently small $\varepsilon>0$.
The variational characterization of the first
eigenvalue therefore yields
\[
\lambda_1(L_{g_\varepsilon})
\leq
\frac{n(n-2)\operatorname{Vol}(S^n,g_0)}
{4c'\log(1/\varepsilon)},
\]
which completes the proof.
\end{proof}

\begin{proposition} \label{prop:dirac_gap}
For the family of metrics $g_\epsilon$ defined above, we have
$$ \lambda_1(D_{g_\epsilon})^2 \ge \frac{1}{4}. $$
\end{proposition}
\begin{proof}
Consider the singular conformal metric $g_* = r^{-2} g_0$ on $S^n \setminus \Sigma$, where we recall that $r$ denotes the extrinsic distance to $\Sigma$. First, we prove that $(S^n \setminus \Sigma, g_*)$ is isometric to the Riemannian product $\mathbb{H}^{n-1} \times S^1$. Using the Fermi coordinates around $\Sigma$ introduced previously, the round metric is written as
$$g_0 = d\rho_0^2 + \cos^2(\rho_0) g_{S^{n-2}} + \sin^2(\rho_0) d\theta^2.$$
Since $r = \sin(\rho_0)$, the singular metric is given by 
$$g_* = \frac{d\rho_0^2}{\sin^2\rho_0} + \cot^2(\rho_0) g_{S^{n-2}} + d\theta^2.$$ Introducing the new variable $\rho$ such that $\sinh \rho = \cot \rho_0$ yields
$$ g_* = d\rho^2 + \sinh^2\rho \, g_{S^{n-2}} + d\theta^2 = g_{\mathbb{H}^{n-1}} + g_{S^1}. $$
Thus, the singular space $(S^n \setminus \Sigma, g_*)$ is indeed globally isometric to $\mathbb{H}^{n-1} \times S^1$.
Now the idea is to use the decomposition of the square of the Dirac operator for a Riemannian product  $$D_{g_*}^2 = D_{\mathbb{H}^{n-1}}^2 \otimes 1 + 1 \otimes D_{S^1}^2.$$ 
The spin structure induced by the standard spin structure of $S^n$ on the normal circle $S^1$ is the bounding spin structure. Spinors therefore satisfy $\psi(\theta+2\pi) = -\psi(\theta)$, meaning $\mathrm{Spec}(D_{S^1}) = \{k + 1/2 : k \in \mathbb{Z}\}$ and in particular 
\begin{equation}\label{ineq: spect gap}
    \lambda_1(D_{S^1}^2) \ge 1/4.
\end{equation}
Since $D_{\mathbb{H}^{n-1}}^2$ is non-negative, for any compactly supported smooth spinor $\Psi \in C_c^\infty(\Sigma_{g_*}(S^n \setminus \Sigma))$, we have:
$$ \int_{S^n \setminus \Sigma} |D_{g_*} \Psi|^2 \, dv_{g_*} \ge \frac{1}{4} \int_{S^n \setminus \Sigma} |\Psi|^2 \, dv_{g_*}. $$
Let $\psi$ be a smooth spinor on $\mathbb S^n$ and set
$\Psi=r^{(n-1)/2}\psi$ under the conformal identification.
Since $\Sigma$ has codimension two, conformal covariance gives
$\Psi,D_{g_*}\Psi\in L^2(g_*)$.
Choose smooth cutoffs $\eta_\delta$ which vanish for $r\leq\delta$,
equal one for $r\geq2\delta$, and satisfy
$|d\eta_\delta|_{g_0}\leq C/\delta$. Then
\[
 \|d\eta_\delta\cdot_{g_*}\Psi\|_{L^2(g_*)}^2
 =\int_{\mathbb S^n}
 r|d\eta_\delta|_{g_0}^2|\psi|_{g_0}^2\,dv_{g_0}
 =O(\delta).
\]
The Leibniz rule therefore gives
$\eta_\delta\Psi\to\Psi$ in the graph norm of $D_{g_*}$,
so the preceding inequality also holds for $\Psi$.

Now let $g_f = f^2 g_0$ be an arbitrary conformal metric. Under the conformal identification of spinor bundles, if $\tilde{\psi} = f^{-(n-1)/2}\psi$, the conformal covariance of the Dirac operator yields $D_{g_f} \tilde{\psi} = f^{-(n+1)/2} D_{g_0} \psi$. Thus, the Rayleigh quotient of $\lambda_1(D_g)^2$ satisfies
$$ \frac{\int_{S^n} |D_{g_f} \tilde{\psi}|^2 \, dv_{g_f}}{\int_{S^n} |\tilde{\psi}|^2 \, dv_{g_f}} = \frac{\int_{S^n} f^{-1} |D_{g_0} \psi|^2 \, dv_{g_0}}{\int_{S^n} f |\psi|^2 \, dv_{g_0}}. $$
For the specific choice $f = f_\epsilon$, we have the pointwise bounds $f_\epsilon^{-1} = \sqrt{r^2 + \epsilon^2} \ge r$ and $f_\epsilon \le r^{-1}$. It follows that
$$ \frac{\int_{S^n} f_\epsilon^{-1} |D_{g_0} \psi|^2 \, dv_{g_0}}{\int_{S^n} f_\epsilon |\psi|^2 \, dv_{g_0}} \ge \frac{\int_{S^n} r |D_{g_0} \psi|^2 \, dv_{g_0}}{\int_{S^n} r^{-1} |\psi|^2 \, dv_{g_0}}. $$
 The right-hand side is exactly the Rayleigh quotient of $D_{g_*}^2$ written in the background metric $g_0$. Thanks to the spectral gap \eqref{ineq: spect gap}, we obtain:
$$ \frac{\int_{S^n} r |D_{g_0} \psi|^2 \, dv_{g_0}}{\int_{S^n} r^{-1} |\psi|^2 \, dv_{g_0}} \ge \frac{1}{4}. $$
Taking the infimum over all non-zero spinors $\tilde{\psi}$ and applying the min-max characterization yields 
$$ \lambda_1(D^2_{g_\epsilon}) \ge \frac{1}{4}. $$
Since $\lambda_1(D_{g_\epsilon})^2 \geq  \lambda_1(D^2_{g_\epsilon})$, we deduce that 
$$ \lambda_1(D_{g_\epsilon})^2 \ge \frac{1}{4}. $$
\end{proof}

\subsection{Proof of Theorem \ref{thm:unbounded}}
We conclude this section by using the estimates above to finally prove Theorem \ref{thm:unbounded}.
\begin{proof}
Let $(S^n, g_0)$, with $n \ge 3$, be the round sphere with its standard spin structure, and
consider the family of conformal metrics $g_\epsilon$ defined above. By Proposition \ref{prop:dirac_gap} and Proposition \ref{prop:laplacian_collapse}, we have
$$ \mathcal{H}(g_\epsilon) = \frac{\lambda_1(D_{g_\epsilon})^2}{\lambda_1(L_{g_\epsilon})} \ge \frac{\log(1/\epsilon)}{4C}. $$
As $\epsilon \to 0$, the ratio $\mathcal{H}(g_\epsilon) \to +\infty$. This completes the proof.
\end{proof}

\bigskip


\subsection*{Acknowledgements}

The author is thankful to David Tewodrose for his support, remarks and suggestions that improved the clarity of the manuscript. This research is supported by the Research Foundation– Flanders (FWO) via the Odysseus II programme no. G0DBZ23N.

\bibliographystyle{plain}
\bibliography{einstein_dirac}

@article{BourguignonGauduchon1992,
  author = {Bourguignon, Jean-Pierre and Gauduchon, Paul},
  title = {Spineurs, op{\'e}rateurs de {D}irac et variations de m{\'e}triques},
  journal = {Communications in Mathematical Physics},
  volume = {144},
  number = {3},
  pages = {581--599},
  year = {1992},
  doi = {10.1007/BF02099184}
}

@article{FriedrichKim2000,
  author = {Friedrich, Thomas and Kim, Eui Chul},
  title = {The {E}instein--{D}irac equation on {R}iemannian spin manifolds},
  journal = {Journal of Geometry and Physics},
  volume = {33},
  number = {1--2},
  pages = {128--172},
  year = {2000},
  doi = {10.1016/S0393-0440(99)00043-1},
  eprint = {math/9905095},
  archivePrefix = {arXiv}
}

@article{Hijazi1986,
  author = {Hijazi, Oussama},
  title = "{A conformal lower bound for the smallest eigenvalue of the {D}irac operator and Killing spinors}",
  journal = {Communications in Mathematical Physics},
  volume = {104},
  number = {1},
  pages = {151--162},
  year = {1986},
  doi = {10.1007/BF01210797}
}

@article{KarpukhinMetrasPolterovich2024,
  author = {Karpukhin, Mikhail and M{\'e}tras, Antoine and Polterovich, Iosif},
  title = {Dirac eigenvalue optimisation and harmonic maps to complex projective spaces},
  journal = {International Mathematics Research Notices},
  volume = {2024},
  number = {21},
  pages = {13758--13784},
  year = {2024},
  doi = {10.1093/imrn/rnae216},
  eprint = {2308.07875},
  archivePrefix = {arXiv}
}

@article{PetridesTewodrose2024,
  author        = {Petrides, Romain and Tewodrose, David},
  title         = {Critical metrics of eigenvalue functionals
                   via {Clarke} subdifferential},
  journal       = {Annales de l'Institut Fourier},
  note          = {To appear},
  eprint        = {2403.07841},
  year={2026},
  archivePrefix = {arXiv}
}

@article{Wood1997,
  author  = {Wood, C. M.},
  title   = {Harmonic sections and equivariant harmonic maps},
  journal = {Manuscripta Mathematica},
  volume  = {94},
  pages   = {1--13},
  year    = {1997}
}

@article{BenyounesLoubeauWood2007,
  author  = {Benyounes, M. and Loubeau, E. and Wood, C. M.},
  title   = {Harmonic sections of {R}iemannian vector bundles, and metrics of {C}heeger--{G}romoll type},
  journal = {Differential Geometry and its Applications},
  volume  = {25},
  number  = {3},
  pages   = {322--334},
  year    = {2007}
}

@article{Alexandrov2007,
  author  = {Alexandrov, Bogdan},
  title   = "{The first eigenvalue of the Dirac operator on locally reducible Riemannian manifolds}",
  journal = {Journal of Geometry and Physics},
  volume  = {57},
  number  = {2},
  pages   = {467--472},
  year    = {2007}
}

@article{Ammann2003,
  author  = {Ammann, Bernd},
  title   = {A spin-conformal lower bound of the first positive
             {Dirac} eigenvalue},
  journal = {Differential Geometry and its Applications},
  volume  = {18},
  number  = {1},
  pages   = {21--32},
  year    = {2003},
  doi     = {10.1016/S0926-2245(02)00095-5}
}

@article{AmmannGrosjeanHumbertMorel2008,
  author  = {Ammann, Bernd and Grosjean, Jean-Fran{\c c}ois
             and Humbert, Emmanuel and Morel, Bertrand},
  title   = {A spinorial analogue of {Aubin}'s inequality},
  journal = {Mathematische Zeitschrift},
  volume  = {260},
  number  = {1},
  pages   = {127--151},
  year    = {2008},
  doi     = {10.1007/s00209-007-0266-5}
}

@article{AmmannHumbertMorel2006,
  author  = {Ammann, Bernd and Humbert, Emmanuel and Morel, Bertrand},
  title   = {Mass endomorphism and spinorial {Yamabe} type problems
             on conformally flat manifolds},
  journal = {Communications in Analysis and Geometry},
  volume  = {14},
  number  = {1},
  pages   = {163--182},
  year    = {2006},
  doi     = {10.4310/CAG.2006.v14.n1.a7}
}

@article{eddaoudi2026harmoniceinsteinmapscriticalmetrics,
      title="{Harmonic-Einstein maps and critical metrics for eigenvalues of the conformal Laplacian}", 
      author={Mehdi Eddaoudi},
      year={2026},
      eprint={2610.28894},
      journal={arXiv},
      primaryClass={math.SP},
      url={https://arxiv.org/abs/2610.28894}, 
}

@article{Martynyuk2026,
  author        = {Martynyuk, Pavel},
  title         = {Conformally critical metrics and optimal bounds
                   for {Dirac} eigenvalues on spin surfaces},
  year          = {2026},
  eprint        = {2604.14840},
  journal = {arXiv},
  primaryClass  = {math.DG},
  url           = {https://arxiv.org/abs/2604.14840}
}

@article{Raulot2009,
  author  = {Raulot, Simon},
  title   = {A {Sobolev}-like inequality for the {Dirac} operator},
  journal = {Journal of Functional Analysis},
  volume  = {256},
  number  = {5},
  pages   = {1588--1617},
  year    = {2009},
  doi     = {10.1016/j.jfa.2008.11.007}
}

@article{AmmannHumbert2006Overview,
  author  = {Ammann, Bernd and Humbert, Emmanuel},
  title   = {The first eigenvalue of the {Dirac} operator
             in a conformal class},
  journal = {International Journal of Geometric Methods
             in Modern Physics},
  volume  = {3},
  number  = {5--6},
  pages   = {833--844},
  year    = {2006},
  doi     = {10.1142/S0219887806001533}
}

@article{AmmannHumbert2006Tori,
  author  = {Ammann, Bernd and Humbert, Emmanuel},
  title   = {The first conformal {Dirac} eigenvalue
             on 2-dimensional tori},
  journal = {Journal of Geometry and Physics},
  volume  = {56},
  number  = {4},
  pages   = {623--642},
  year    = {2006},
  doi     = {10.1016/j.geomphys.2005.04.007}
}

@article{Branding2019,
  author  = {Branding, Volker},
  title   = {Nonlinear {Dirac} equations, monotonicity formulas
             and {Liouville} theorems},
  journal = {Communications in Mathematical Physics},
  volume  = {372},
  number  = {3},
  pages   = {733--767},
  year    = {2019},
  doi     = {10.1007/s00220-019-03608-z}
}

@article{Ammann2009,
  author  = {Ammann, Bernd},
  title   = {The smallest {Dirac} eigenvalue in a spin-conformal class
             and {CMC} immersions},
  journal = {Communications in Analysis and Geometry},
  volume  = {17},
  number  = {3},
  pages   = {429--479},
  year    = {2009},
  doi     = {10.4310/CAG.2009.v17.n3.a2}
}

@misc{SireXu2021,
  author        = {Sire, Yannick and Xu, Tian},
  title         = {On the {B\"ar--Hijazi--Lott} invariant for the
                   {Dirac} operator and a spinorial proof of the
                   {Yamabe} problem},
  year          = {2021},
  eprint        = {2112.03640},
  archivePrefix = {arXiv},
  primaryClass  = {math.DG}
}

@article{JulioBatalla2025,
  author  = {Julio-Batalla, Jurgen},
  title   = {Spinorial {Yamabe}-Type Equations and the
             {B\"ar--Hijazi--Lott} Invariant},
  journal = {The Journal of Geometric Analysis},
  volume  = {35},
  pages   = {227},
  year    = {2025},
  doi     = {10.1007/s12220-025-02050-z}
}

@article{Lott1986,
  author  = {Lott, John},
  title   = {Eigenvalue bounds for the {Dirac} operator},
  journal = {Pacific Journal of Mathematics},
  volume  = {125},
  number  = {1},
  pages   = {117--126},
  year    = {1986}
}

@article{IsobeXu2023,
  author  = {Isobe, Takeshi and Xu, Tian},
  title   = {Solutions of spinorial {Yamabe}-type problems on
             {$S^m$}: Perturbations and applications},
  journal = {Transactions of the American Mathematical Society},
  volume  = {376},
  number  = {9},
  pages   = {6397--6446},
  year    = {2023},
  doi     = {10.1090/tran/8961}
}

@article{Soler1970,
  author  = {Soler, Mario},
  title   = {Classical, Stable, Nonlinear Spinor Field
             with Positive Rest Energy},
  journal = {Physical Review D},
  volume  = {1},
  pages   = {2766--2769},
  year    = {1970},
  doi     = {10.1103/PhysRevD.1.2766}
}

@article{GrosjeanGinoux,
     author = {Nicolas Ginoux and Jean-Fran\c{c}ois Grosjean},
     title = {Almost harmonic spinors},
     journal = {Comptes Rendus. Math\'ematique},
     pages = {811--814},
     year = {2010},
     publisher = {Elsevier},
     volume = {348},
     number = {13-14},
     doi = {10.1016/j.crma.2010.06.010},
     language = {en},
}

@article{Friedrich1980,
  author  = {Friedrich, Thomas},
  title   = {Der erste {E}igenwert des {D}irac-{O}perators einer kompakten,
         {R}iemannschen {M}annigfaltigkeit nichtnegativer
         {S}kalarkr{\"u}mmung},
  journal = {Mathematische Nachrichten},
  volume  = {97},
  number  = {1},
  pages   = {117--146},
  year    = {1980},
  doi     = {10.1002/mana.19800970111}
}

@article{Baer1992,
  author  = {B{\"a}r, Christian},
  title   = {Lower eigenvalue estimates for {D}irac operators},
  journal = {Mathematische Annalen},
  volume  = {293},
  number  = {1},
  pages   = {39--46},
  year    = {1992},
  doi     = {10.1007/BF01444701}
}

@book{Ginoux2009,
  author    = {Ginoux, Nicolas},
  title     = {The {D}irac Spectrum},
  series    = {Lecture Notes in Mathematics},
  volume    = {1976},
  publisher = {Springer},
  address   = {Berlin, Heidelberg},
  year      = {2009},
  doi       = {10.1007/978-3-642-01570-0},
  url       = {https://link.springer.com/book/10.1007/978-3-642-01570-0}
}

\end{document}